\documentclass[a4paper,10pt]{amsart}
\usepackage{amssymb,amsthm,amsmath}
\usepackage{ifthen}
\usepackage{graphicx}
\nonstopmode
 \numberwithin{equation}{section}
\newtheorem{thm}{Theorem}[section]
\newtheorem{cor}[thm]{Corollary}
\newtheorem{lem}[thm]{Lemma}
\newtheorem{prop}[thm]{Proposition}
\newtheorem{defn}[thm]{Definition}

\theoremstyle{definition}
\newtheorem{rmk}[thm]{Remark}
\newtheorem{example}[thm]{Example}

{\qed\bigskip}

\newcounter{alphabet}
\newcounter{tmp}

\ifx\undefined\bysame
\newcommand{\bysame}{\leavevmode\hbox to3em{\hrulefill}\,}
\fi

\begin{document}
\baselineskip=21pt
\markboth{} {}

\bibliographystyle{amsplain}
\title[Behavior of Gabor frame operators on Wiener amalgam spaces]
{Behavior of Gabor frame operators on Wiener amalgam spaces}

\author{Anirudha poria} 

\address{Department of Mathematics,
Indian Institute of Technology Guwahati,
Guwahati 781039, \;\;  India.}
\email{a.poria@iitg.ernet.in}

\begin{abstract} 
It is well known that the Gabor expansions converge to identity operator in weak* sense on the Wiener amalgam spaces as sampling density tends to infinity. In this paper we prove the convergence of Gabor expansions to identity operator in the operator norm as well as weak* sense on $W(L^p, \ell^q)$ as the sampling density tends to infinity. Also we show the validity of the Janssen's representation and the Wexler-Raz biorthogonality condition for Gabor frame operator on $W(L^p, \ell^q)$.
\end{abstract}
\date{\today}
\maketitle
\def\BC{{\mathbb C}} \def\BQ{{\mathbb Q}}
\def\BR{{\mathbb R}} \def\BI{{\mathbb I}}
\def\BZ{{\mathbb Z}} \def\BD{{\mathbb D}}
\def\BP{{\mathbb P}} \def\BB{{\mathbb B}}
\def\BS{{\mathbb S}} \def\BH{{\mathbb H}}
\def\BE{{\mathbb E}}
\def\BN{{\mathbb N}}
\def\LP{{W(L^p, \ell^q)}}
\def\L1{{W(\mathbb{R}^d)}}
\vspace{-.5cm}

\section{Introduction and the main result}
We begin with some elementary definition. The time-frequency shift $ \tau (t,\omega ) $ for functions $g$ on $ \BR^d $ is define by \[ \left(\tau (t,\omega )g \right)(x)=g(x-t) e^{i 2 \pi \langle x , \omega \rangle } , \;\;\; t,\omega \in \BR^d. \]

The windowed Fourier transform of $f \in L^2(\BR^d)$ with respect to $g \in L^2(\BR^d)$ is defined by \begin{equation} \label{eq1.1}
(F_g f)(t,\omega )= \langle f,\tau (t,\omega )g  \rangle . 
\end{equation} 

The inversion formula for windowed Fourier transform is given by
  \begin{equation} \label{eqi}
f= \frac{1}{\langle \gamma, g \rangle} \iint\limits_{\BR^{2d}} (F_g f)(t,\omega) \tau (t,\omega )\gamma dtd{\omega},
\end{equation} where  $\gamma \in L^2(\BR^d) $ satisfies $ \langle \gamma,g \rangle \neq 0 $ and the integral is convergent in $ L^2(\BR^d) $ norm (e.g., see \cite{gro}, p. 48).

Given a window $g \in L^2(\BR^d)$ and $a,b>0$, the collection $ \{ \tau (na,mb)g :m,n \in \BZ^d \} $ is a Gabor frame for $ L^2(\BR^d) $ if there exist constants $ A,B>0$ (called frame bounds) such that \[ A \Vert f \Vert_2^2 \leq \sum_{n,m \in \BZ^d} |\langle f,\tau (na,mb)g \rangle|^2 \leq B \Vert f \Vert_2^2, \;\;\; \forall f \in L^2 (\BR^d).  \]

If $ \{ \tau (na,mb)g :m,n \in \BZ^d \} $ is a Gabor frame then there exists a dual window $ \gamma \in L^2 (\BR^d) $ such that $ \{ \tau (na,mb)\gamma :m,n \in \BZ^d \} $ is also a Gabor frame for $ L^2 (\BR^d) $ and 
\begin{eqnarray} \label{eq1.2}
f &=& \sum_{n,m \in \BZ^d} \langle f,\tau (na,mb)g \rangle \tau (na,mb)\gamma\\
&=& \sum_{n,m \in \BZ^d} \langle f,\tau (na,mb)\gamma \rangle \tau (na,mb)g , \;\;\; \forall f \in L^2(\BR^d). \nonumber
\end{eqnarray}

The series in (\ref{eq1.2}) converge unconditionally in $L^2$, and
by the definition of frame, the $\ell^2 $-norm of the sequence of Gabor coefficients $ \{\langle f,\tau (na,mb)g \rangle \}$ is an equivalent norm for $L^2(\BR^d)$. We refer to \cite{gro} for a detailed study.\\
Define \begin{eqnarray} \label{eqs}
S_{a,b;g,\gamma}f = \frac{{(ab)}^d}{\langle \gamma,g \rangle} \sum_{n,m \in \BZ^d} \langle f,  \tau (na,mb)g \rangle  \tau (na,mb) \gamma.
\end{eqnarray}

Then $S_{a,b;g,\gamma}f$ can be regarded as a Riemannian sum of the integral in (\ref{eqi}). We may expect that it is well defined for every $f \in L^2(\BR^d)$, and that it converges to $f$ whenever $a$ and $b$ tend to zero. Weisz \cite{wei} proved that this is the case if both $g$ and $\gamma $ are in $ S_0(\BR^d) := \{g:F_g g \in L^1(\BR^{2d}) \} $, and he proved the convergence in various norms. Also it is well known that whenever $g \in S_0(\BR^d)$ and defines a Gabor frame, then the dual window is also in $S_0(\BR^d)$ (see \cite{lei1, lei2}).

When we consider $S_{a,b;g,\gamma}$ as a frame operator, the convergence of $S_{a,b;g,\gamma}f$ implies that $ \{ \tau (na,mb)g :m,n \in \BZ^d \} $ is asymptotically close to a tight frame. If $g,\gamma$ are in a bigger space than $ S_0(\BR^d)$ then the result of Weisz \cite{wei} does not imply the convergence of $S_{a,b;g,\gamma}f$ as $(a,b)$ tends to $(0,0)$. For example if $g=\gamma=\chi_{[0,1]}$, it is not clear whether $S_{a,b;g,\gamma}f$ is convergent to $f$ as $(a,b)$ tends to $(0,0)$ since neither $g$ nor $\gamma$ is in $ S_0(\BR^d) $. In \cite{sun} Sun gave a weaker condition on $g$ and $\gamma$ for $S_{a,b;g,\gamma}f$ to be convergent. He proved that if both $g$ and $\gamma$ are in the Wiener space
 \[ \L1 :=  \bigg\{ g: g \; \mathrm{is \; measurable \; and \; }  \Vert g \Vert_{\L1} := \sum\limits_{k \in \BZ^d} \Vert g \cdot T_k \chi_{[0,1)^d} \Vert_{\infty} < \infty \bigg\}, \] then $S_{a,b;g,\gamma}f$ is convergent to $ f $ on $ L^p(\BR^d) $ as $(a,b)$ tends to $(0,0)$. Note that $S_0(\BR^d)$ is a proper subspace of $\L1$ (see \cite{zim}, Theorem 3.2.13).

In this paper we extend Sun's \cite{sun} results and show if both both $g$ and $\gamma$ are in the Wiener space then $S_{a,b;g,\gamma}f$ converges to $ f $ on amalgam space $ \LP $, $1\leq p,q < \infty$,  as $(a,b)$ tends to $(0,0)$. 
Our main result is the following theorem.
\begin{thm}\label{thm1}
Let $ g,\gamma \in \L1$.Then we have:\\
$(i)$ For any $f \in \LP$, $ 1 \leqslant p,q < \infty, $ 
\begin{equation} \label{eq1} \lim_{(a,b)\rightarrow (0,0)} \Vert S_{a,b;g,\gamma}f -f \Vert_{\LP} =0 \end{equation} and conclusion holds for $q= \infty$ but fails if $ p=\infty .$\\
$(ii)$ Moreover, if $\overline{g} \cdot \gamma$ is locally Riemann integrable, then for any  $ 1 \leqslant p,q \leqslant \infty, $ 
\begin{equation} 
\lim_{(a,b)\rightarrow (0,0)} \Vert S_{a,b;g,\gamma} -I \Vert_{\LP \rightarrow \LP} =0. \end{equation} 
\end{thm}

Recent results (\cite{por}) have shown weak convergence of Hilbert space valued Gabor expansions on weighted amalgam spaces. So if we consider Hilbert space valued Gabor expansions then whether it converge to identity operator in the operator norm, seems to be a reasonable question. Also same question one can ask for modulation spaces as we know weak convergence of Gabor expansions on modulation spaces (\cite{wei}). For the time being we leave these two questions as open. In section 5, we reinvestigated two known results: the Janssen's representation and the Wexler-Raz biorthogonality on $\LP$ for Gabor frame operator and these are two important results in this context.

The paper is organized as follows. In section 2, we provide necessary background and some of the properties of Wiener amalgam spaces. In section 3, we provide some important results of Gabor frame operators on $L^2(\BR^d)$. In section 4, we prove the strong and weak* convergence result for Gabor frame operators on $\LP$, $1 \leq p,q < \infty .$ For $p=\infty$ we produce a counter example showing that (\ref{eq1}) does not hold in this case. Finally in section 5, we show that Janssen's representation and the Wexler-Raz biorthogonality condition hold for Gabor frame operator $S_{a,b;g,\gamma}$ on $W(L^p, \ell^q)$.

\section{Wiener amalgam spaces}
Let fix $d \geq 1$, $d \in \BN$. For a set $\mathbb{X} \neq \emptyset $ let $ \mathbb{X}^d $ be its Cartesian product $ \mathbb{X}\times ...\times \mathbb{X} $ taken with itself $d$-times. For $ x=(x_1 ,...,x_d) \in \BR^d $ ~\mbox{and}~ $ y=(y_1 ,...,y_d) \in \BR^d $ set \[x \cdot y := \sum_{k=1}^d x_k y_k  \;\;\;\;\; and \;\;\;\; |x|:= \max_{k=1,...,d} |x_k|. \]

The $\ell_p $ $(1\leq p \leq \infty)$ space consists of all complex sequences $a=(a_k)_{k \in \BZ^d}$ for which \[ \Vert a \Vert_{\ell_p} := \left(\sum_{k \in \BZ^d} |a_k|^p \right)^{\frac{1}{p}} \;\;\; < \infty \] with the usual modification if $ p= \infty $.

We briefly write $ L_p$ or $L_p(\BR^d)$ instead of $L_p(\BR^d , \lambda)$ space equipped with the norm (or quasi-norm) $\Vert f \Vert_p := \left(\displaystyle\int_{\BR^d} |f|^p d\lambda \right)^{\frac{1}{p}} $, $(0 < p \leq \infty),$ where $ \lambda $ is the Lebesgue measure. 
The space of continuous functions on $\BR^d$ with the supremum norm is denoted by $C(\BR^d).$\\
Translation and modulation of a function $f$ are defined, respectively, by
\[ T_x f(t):=f(t-x) \;\; and \;\;\; M_\omega f(t):= e^{2 \pi i \langle\omega , t\rangle} f(t)\;\;\; (x,\omega \in \BR^d), \] where $i=\sqrt{-1}. $

Let $Q$ denote the unit cube $[0,1)^d $ and $Q_a=[0,a)^d$. The characteristic function of a measurable set $E$ is $\chi_E$. A measurable function $f$ belongs to the Wiener amalgam space $\LP $ $(1 \leqslant p,q \leqslant \infty )$ if \[ \Vert f \Vert_{\LP} := \left(\sum_{k \in \BZ^d} \Vert f \cdot T_k \chi_Q \Vert_p^q \right)^{\frac{1}{q}} \;\; < \;\; \infty , \] with the obvious modification for $q= \infty .$ The closed subspace of $ W(\BR^d)$ containing continuous functions is denoted by $ W(C, \ell^1)$ and called the Wiener algebra. It is easy to see that $ W(L^p, \ell^p) = L^p(\BR^d) ,$ \[ W(L^{p_1}, \ell^q) \hookleftarrow W(L^{p_2} , \ell^q) \;\; \;\;\;(p_1 \leq p_2) \] and
\[ W(L^p, \ell^{q_1}) \hookrightarrow W(L^p, \ell^{q_2}) \;\;\;\;\; (q_1 \leq q_2 ) \; , \] $ (1 \leq p_1 , p_2 , q_1 ,q_2 \leq \infty  ).$ Thus \[ W(L^{\infty}, \ell^1) \subset L^p (\BR^d) \subset W(L^1, \ell^{\infty}) \;\;\; (1 \leq p \leq \infty). \]

The K\"othe dual of $\LP$ is the space of all measurable functions
$g$ on $\BR^d$ such that $g \cdot \LP \subseteq L^1(\BR^d).$ It is equal to $W(L^{p'}, \ell^{q'}),$ where $1/p+1/p'=1/q+1/q'=1$ for all $1 \leq p, q \leq \infty $. The pairing 
 \[ \langle \cdot, \cdot \rangle : \LP \times W(L^{p'}, \ell^{q'}) \rightarrow \BC, \;\;\;\;\; \langle f, g \rangle= \int_{\BR^d}  f(x) \overline{g(x)} dx, \] is bounded. The collection of all bounded linear operators from $\LP$ to $\LP$ is denoted by $B(\LP)$.
 
The dual and K\"othe dual of the amalgam spaces are given in the next lemma.
\begin{lem}
Let $1/p+1/p'=1/q+1/q'=1$. Then\\
a. For $1 \leq p, q < \infty $, the dual space of $\LP$ is $ W(L^{p'}, \ell^{q'}).$\\
b. For $1 \leq p, q \leq \infty $ the K\"othe dual of $\LP$ is $ W(L^{p'}, \ell^{q'}).$
\end{lem}
We refer to the paper of Fournier and Stewart \cite{fou} for detailed study on classical amalgam spaces.

\section{Gabor frame operators on $L^2(\BR^d)$}
In this section, we study some known result of convergence of $S_{a,b;g,g}f$ for $f \in L^2(\BR^d).$ 
When $g=\gamma$ then we simply write $S_{a,b}$ instead of $S_{a,b;g,g}$, i.e.,
\begin{equation} \label{eq3.1}
 S_{a,b}f=\frac{(ab)^d}{\Vert g \Vert_2^2} \sum_{n,m \in \BZ^d} \langle f,  \tau (na,mb)g \rangle  \tau (na,mb)g.
 \end{equation}
Recall that a sequence $\{ \phi_n : n \in \BZ \}$ in a Hilbert space $\mathcal{H}$ is said to be a Bessel sequence if 
 \[ \sum_{n \in \BZ} |\langle f,\phi_n \rangle|^2 < \infty, \;\;\; \forall f \in  \mathcal{H}.\]
 We refer to \cite{gro} for frames and Bessel sequences. It was shown that $\{ \phi_n : n \in \BZ \}$ is a Bessel sequence if and only if there exists some constant $ M < + \infty $ such that \[ \sum_{n \in \BZ} |\langle f,\phi_n \rangle|^2 < M \Vert f \Vert^2 , \;\;\; \forall f \in  \mathcal{H}. \]
 The constant $M$ is called the upper frame bound for the Bessel sequence.
We state the following result (see \cite{sun}) without proof, which will use in next section.
\begin{thm} \label{thm2}
Let $g \in L^2(\BR^d)$ and $S_{a,b}$  be define as in (\ref{eq3.1}), where $a,b >0$. Then the following assertions are equivalent.\\
$(i)$ For any $f \in L^2(\BR^d)$, there exist constants $ a_f , b_f > 0 $ such that $ S_{a,b} f $ is well defined whenever $0<a<a_f$ and $ 0<b<b_f $ and the limit $ \lim\limits_{(a,b) \rightarrow (0,0)} S_{a,b} f$ exists in the $L^2(\BR^d)$  sense.\\
$(ii)$  $ S_{a,b} $ is well defined on $L^2(\BR^d)$ for $a,b \in (0,1]$ and there exists some constant $M<+ \infty $ such that \[ \Vert S_{a,b} \Vert \leq M , \;\;\; \forall a,b \in (0,1]. \] 
$(iii)$ There is some constant $M<+ \infty $ such that $ \{(ab)^{d/2} \Vert g \Vert_2^{-1} \tau (na,mb)g : n,m \in \BZ^d  \} $ is a
Bessel sequence with upper frame bound $M$, $ \forall a,b \in (0,1]. $
\end{thm}
For the case of $ g \neq \gamma, $ we have the following result.

\begin{thm} \label{thm3}
Let $g,\gamma \in L^2(\BR^d)$ be such that $\langle \gamma,g \rangle \neq 0.$ Let $ S_{a,b;g,\gamma} $ be defined as in (\ref{eqs}), where $a,b>0$. Suppose that there is some constant $M< + \infty $ such that \[ \Vert S_{a,b;g,g} \Vert \leq M \;\;\; and \;\;\; \Vert S_{a,b;\gamma,\gamma} \Vert \leq M ,\;\;\; \forall \; 0<a,b \leq 1. \] Then $ S_{a,b;g,\gamma} f $ is well defined for any $ f \in L^2(\BR^d)$ and \[ \lim_{(a,b) \rightarrow (0,0)} S_{a,b;g,\gamma} f = f , \;\; \forall f \in L^2(\BR^d). \]
\end{thm}
\section{Gabor frame operators on $\LP$}
It was shown in \cite{gro, hei, wal} that $S_{a,b;g,\gamma}$ is bounded operator form $L^p(\BR^d)$ to $L^p(\BR^d)$ for any $1\leqslant p\leqslant \infty$ whenever $g$ and $\gamma$ are in $\L1$. Further this result has been proved for amalgam spaces (see \cite{fei,oko}). In this section we prove the strong and weak* convergence result for Gabor frame operators $S_{a,b;g,\gamma}$ on $\LP$, $1 \leq p,q < \infty $ whenever $g,\gamma$ are in $\L1$. 

Before proving Theorem 1.1 we state the following result on the Wiener space $\L1$ which can be found in \cite{gro}.

\begin{prop} \label{com1}
 If $g \in \L1$ and $a>0$, then 
 \[\sum_{n \in \mathbb{Z}^d}|g(x-an)| \leq \left(1+\frac{1}{a}\right)^d\; {\parallel g\parallel}_{W(\BR^d)} ,\;\;a.e.\]
 \end{prop}
 For $g,\gamma \in \L1$ and $a,b>0$, define\
\[ G_{a,b;n}(x)= \sum_{k\in {\BZ}^d}{ \overline{g}\left(x-\frac{n}{b}-ak\right) \gamma(x-ak),\;\; n \in {\BZ}^d}. \]
\begin{lem}\label{lem1} 
$($see \cite{sun}, Lemma 3.3 $)$
For any $g,\gamma \in \L1$, we have \begin{equation} \label{eq4.1} \sum_{n \in {\BZ}^d} {\|G_{a,b;n}\|}_{\infty}\leqslant \left(1+\frac{1}{a} \right)^d (2+2b)^d \|g\|_\L1 \|\gamma\|_\L1 , \;\;\forall a,b>0, \end{equation} and \begin{equation}\lim_{(a,b)\rightarrow (0,0)}{\sum_{n \in {\BZ}^d \setminus\{0\}}}{{a^d \|G_{a,b;n}\|}_{\infty}} = 0.\end{equation}
\end{lem}
Using the Walnut representation of the Gabor frame operator on $\LP$ $($see \cite{oko}$)$ with an appropriate modification we get the following proposition.
\begin{prop} \label{com2}
Let $g$,$\gamma$ $\in$ $\L1$ and let $a,b> 0$. Then the operator
\begin{equation}\label{eq10}
(S_{a,b;g,\gamma}f)(x)=\frac{1}{\langle \gamma, g \rangle} \sum_{n \in \mathbb{Z}^d}a^d G_{a,b;n}(x) f\left(x-\frac{n}{b}\right)
\end{equation}
is bounded from $\LP$ to $\LP$, $1$ ${\leq}$ $p,q$ ${\leq}$ ${\infty}$ with operator norm
\[\Vert S_{a,b;g,\gamma} \Vert _ {\LP \rightarrow \LP} \leq \frac{a^d}{|\langle\gamma, g \rangle|} \left(1+\frac{1}{a} \right)^d (2+2b)^d \|g\|_\L1 \|\gamma\|_\L1.\]
\end{prop}
\begin{proof}
If $g,\gamma \in \L1$ then by previous lemma \[\sum_{n \in {\BZ}^d} {\|G_{a,b;n}\|}_{\infty} < \infty\]
If $f \in \LP$ then\
\begin{eqnarray*}
\Vert S_{a,b;g,\gamma}f \Vert _ \LP &=&{\Vert\frac{1}{\langle \gamma, g \rangle} \sum_{n \in \mathbb{Z}^d}a^d G_{a,b;n} T_{\frac{n}{b}}f \Vert}_\LP \\
&\leq & \frac{a^d}{|\langle\gamma, g \rangle|} \sum_{n \in \mathbb{Z}^d} \Vert G_{a,b;n} T_{\frac{n}{b}}f \Vert_\LP \\
 & = & \frac{a^d}{|\langle\gamma, g \rangle|} \sum_{n \in \mathbb{Z}^d}
\left(\sum_{k \in \mathbb{Z}^d} \Vert G_{a,b;n} T_{\frac{n}{b}}f \cdot T_k \chi_Q \Vert_p ^ q \right)^\frac{1}{q} \\
& \leq & \frac{a^d}{|\langle\gamma, g \rangle|} \sum_{n \in \mathbb{Z}^d}\Vert G_{a,b;n} \Vert_{\infty} \left(\sum_{k \in \mathbb{Z}^d} \Vert f \cdot T_{k-\frac{n}{b}} \chi_Q \Vert_p ^q \right)^\frac{1}{q}\\
& = & \frac{a^d}{|\langle\gamma, g \rangle|} \Vert f \Vert_{\LP} \sum_{n \in \mathbb{Z}^d}\Vert G_{a,b;n} \Vert_{\infty}\\
& \leq & C \; \|g\|_\L1 \|\gamma\|_\L1 \Vert f \Vert_{\LP}\\
\end{eqnarray*}
Where  $C = \frac{a^d}{|\langle\gamma, g \rangle|} \left(1+\frac{1}{a} \right)^d (2+2b)^d . $
\end{proof}

Let \begin{equation} \label{ga}
G_a (x):= \frac{a^d}{\langle\gamma, g \rangle} G_{a,b;0} (x) =  \frac{a^d}{\langle\gamma, g \rangle} \sum_{k \in \mathbb{Z}^d} \overline{g}(x-ak)\gamma(x-ak), \;\; x \in {\BR}^d. 
\end{equation}
By proposition \ref{com1}, we have \
\begin{equation} \label{mo}
M_0 := \sup_{0<a \leqslant 1} \Vert G_a -1 \Vert_{\infty} \leqslant \sup_{0<a \leqslant 1} \frac{1}{|\langle\gamma, g \rangle|} (1+a)^d \Vert \overline{g} . \gamma \Vert _{\L1} < \infty.
\end{equation}
The following lemma is the key to Theorem 1.1.
\begin{lem} \label{lem2}
Let $g,\gamma \in \L1$ and $1 \leqslant p,q \leqslant \infty. $ Then\\
$(i)$ For any $f \in \LP,$ \[\lim_{(a,b)\rightarrow (0,0)} (\Vert S_{a,b;g,\gamma}f-f \Vert_\LP -\Vert (G_a -1)f\Vert_\LP )=0 ; \]
$(ii)$ \[\lim_{(a,b)\rightarrow (0,0)} (\Vert S_{a,b;g,\gamma}-I \Vert_{\LP \rightarrow \LP }-\Vert G_a -1\Vert_\infty )=0. \]
\end{lem}
\begin{proof}
Define operators $T_{a;g,\gamma}$ and $R_{a,b;g,\gamma}$ on $\LP$ by
\[T_{a;g,\gamma} f=(G_a -1)f ,\]
\[R_{a,b;g,\gamma} f= \frac{1}{\langle \gamma, g \rangle} \sum_{n \in \mathbb{Z}^d \setminus \{0\}} a^d G_{a,b;n}\cdot f\left(\cdot-\frac{n}{b}\right),\; f \in \LP. \]\
Now we can rewrite the Walnut’s representation as 
\[S_{a,b;g,\gamma}f - f = T_{a;g,\gamma} f + R_{a,b;g,\gamma} f ,\; \forall f \in \LP. \]
Then we have\
\[ \Vert S_{a,b;g,\gamma}f - f \Vert_{\LP} \leqslant \Vert T_{a;g,\gamma} f \Vert_{\LP }+ \Vert R_{a,b;g,\gamma} f \Vert_{\LP} . \]
By Lemma \ref{lem1}, we have 
\[\lim_{(a,b)\rightarrow (0,0)}{\sum_{n \in {\BZ}^d \setminus\{0\}}}{{a^d \|G_{a,b;n}\|}_{\infty}} = 0.\]
Hence \[\lim_{(a,b)\rightarrow (0,0)} {\Vert R_{a,b;g,\gamma} \Vert}_{\LP \rightarrow \LP} \leqslant \lim_{(a,b)\rightarrow (0,0)}{\sum_{n \in {\BZ}^d \setminus\{0\}}}{{a^d \|G_{a,b;n}\|}_{\infty}} =0. \]
On the other hand, it is easy to see that\
\[ \Vert T_{a;g,\gamma} \Vert_{\LP \rightarrow \LP} = \Vert G_a -1 \Vert_{\infty}.\]
\end{proof}
To prove our main result we make use of the following two lemmas.
\begin{lem} \label{lem3} 
$($See \cite{sun}, Lemma 3.5$)$ Suppose that $ f \in \L1 $ is locally Riemann integrable. Then we have\
\[\lim_{a \rightarrow 0} \sup_{y \in {\BR}^d} \bigg| \sum_{n \in {\BZ}^d} a^d f(y+na) - \int\limits_{{\BR}^d} f(x) dx \bigg| =0. \]
\end{lem}

\begin{lem} \label{lem4.5}
$($\cite{sun}$)$ If $ f \in L^p(\BR^d)$ and $ G_a $ is define as in (\ref{ga}) then \[ \lim_{a \rightarrow 0} \Vert (G_a - 1)f \Vert_p =0 \]
\end{lem}
\begin{proof}
First, we consider the case of $p=2$.\\
By Walnut’s representation and (\ref{eq4.1}), it is easy to see that
\[ \Vert S_{a,b;g,g} \Vert \leq \frac{1}{\Vert g \Vert_2^2} (1+a)^d (2+2b)^d \Vert g \Vert_{\L1}^2 \leq \frac{8^d}{\Vert g \Vert_2^2} \Vert g \Vert_{\L1}^2 ,\;\; 0<a,b \leq 1 \] and
\[ \Vert S_{a,b;\gamma,\gamma} \Vert \leq \frac{1}{\Vert \gamma \Vert_2^2} (1+a)^d (2+2b)^d \Vert \gamma \Vert_{\L1}^2 \leq \frac{8^d}{\Vert \gamma \Vert_2^2} \Vert \gamma \Vert_{\L1}^2 ,\;\; 0<a,b \leq 1 \]
Now we see from Theorem \ref{thm3} that
\[ \lim_{(a,b)\rightarrow (0,0)} \Vert S_{a,b;g,\gamma}f-f \Vert_2 =0, \;\; \forall f \in L^2(\BR^d). \]
since $ W(L^2, l^2) = L^2(\BR^d) $ then by using Lemma \ref{lem2} , we have \begin{eqnarray} \label{eq4.5}
\lim_{a \rightarrow 0} \Vert (G_a - 1)f \Vert_2 =0, \;\; \forall f \in L^2(\BR^d). 
\end{eqnarray} 
Take some $ f \in L^p(\BR^d) $, $ 1 \leq p < \infty . $ For any $ \epsilon >0 , $ there is some $ A > 0$ such that \[ \int\limits_{{\Vert x \Vert}_{\infty} > A} |f(x)|^p dx < \epsilon^p . \] And there is some $\delta >0$ such that for any measurable set $ E\subset \BR^d $ with $|E|< \delta ,$
\[ \int\limits_E  |f(x)|^p dx < \epsilon^p . \] Here $ | \cdot | $ use to denote the Lebesgue measure of a measurable set. By (\ref{eq4.5}), we have
\[ \lim_{a \rightarrow 0 } |\{ x \in [-A,A]^d : |G_a (x)-1| \geq \epsilon \}| \leq \lim_{ a \rightarrow 0 } \frac{1}{\epsilon^2} \int\limits_{\BR^d}  |G_a (x)-1|^2 \cdot | \chi_{[ -A,A]^d} (x)|^2 dx =0. \]
Hence, we can find some $0<a_0<1$ such that for any $0<a<a_0$, \[ | \{ x \in [-A,A]^d : |G_a (x)-1| \geq \epsilon \} | <\delta. \]
It follows that \begin{eqnarray*}
\Vert (G_a -1)f \Vert_p^p = \int\limits_{\{ x \in [-A,A]^d : |G_a (x)-1| \geq \epsilon \} }  |(G_a (x)-1)f(x)|^p dx \\ +\int\limits_{\{ x \in [-A,A]^d : |G_a (x)-1| < \epsilon \}} |(G_a (x)-1)f(x)|^p dx \\ + \int\limits_{x \notin [-A,A]^d } |(G_a (x)-1)f(x)|^p dx 
\end{eqnarray*}
\[ \leq {M_0}^p {\epsilon^p} + \epsilon^p \Vert f \Vert_p^p +{ M_0}^p {\epsilon^p} , 0<a<a_0 ,\]
where $M_0$ is defined in (\ref{mo}).Hence \[ \lim_{a \rightarrow 0} \Vert (G_a - 1)f \Vert_p =0 .\]
\end{proof}

Now we are in a position to prove Theorem 1.1 .\\
{\bf Proof of Theorem \ref{thm1}.} Let $ f \in \LP$, $1 \leq p,q < \infty . $ \\
Then for any $ \epsilon > 0 $, there exists some $ N \in \BZ $ such that\[ \sum_{{{\Vert k \Vert}_\infty} > N} {\Vert f \cdot T_k \chi_Q \Vert}_p^q < {\epsilon}^q . \]
\begin{eqnarray} \label{eq5}
{\Vert (G_a -1)f \Vert}_{\LP}^q &=& \sum_{k \in \BZ^d} {\Vert (G_a -1)f \cdot T_k \chi_Q \Vert}_p^q \nonumber\\
&=& \sum_{{{\Vert k \Vert}_\infty} > N} {\Vert (G_a -1)f \cdot T_k \chi_Q \Vert}_p^q \nonumber \\
&\;& + \sum_{{{\Vert k \Vert}_\infty} \leqslant N} {\Vert (G_a -1)f \cdot T_k \chi_Q \Vert}_p^q
\end{eqnarray}
\begin{eqnarray} \label{eq6}
\sum_{{{\Vert k \Vert}_\infty} > N} {\Vert (G_a -1)f \cdot T_k \chi_Q \Vert}_p^q & \leq & {\Vert (G_a -1) \Vert}_{\infty}^q \sum_{{{\Vert k \Vert}_\infty} > N} {\Vert f \cdot T_k \chi_Q \Vert}_p^q \nonumber\\
& \leq & {{M}_0}^q {\varepsilon}^q
\end{eqnarray} 
where $M_0$ is defined in (\ref{mo}.)\\
Since
 \[ {\Vert f \cdot T_k \chi_Q \Vert}_p \leqslant \Vert f \Vert_{\LP} < \infty , \;\; \forall  f \in \LP,  \]  for every $k \in \BZ^d$. So $  f \cdot T_k \chi_Q \in L^p $ for every $k \in \BZ^d$. \\
By using Lemma \ref{lem4.5}, on $  f \cdot T_k \chi_Q \in L^p(\BR^d) $ we get,
\begin{eqnarray} \label{eq7}
{\Vert (G_a -1)f \cdot T_k \chi_Q \Vert}_p^q & \leqslant & ( 2 M_0^p \varepsilon^p +\varepsilon^p \Vert f \cdot T_k \chi_{Q} \Vert_p^p )^{\frac{q}{p}} \nonumber \\
& \leqslant & (2 M_0^p \varepsilon^p + \varepsilon^p \Vert f \Vert_{\LP}^p)^{\frac{q}{p}}\nonumber \\
&=& C \varepsilon^q  
\end{eqnarray}
 where $C = (2 M_0^p + \Vert f \Vert_{\LP}^p)^{\frac{q}{p}} $ and for sufficiently small $a$.\\ 
 Replacing $\varepsilon $ by $ \frac{\varepsilon}{(2N+1)^{\frac{d}{q}}}\;\; $\
and taking summation on both side of last inequality we get,\\
 \begin{eqnarray} \label{eq8}
\sum_{{{\Vert k \Vert}_\infty} \leqslant N} {\Vert (G_a -1)f \cdot T_k \chi_Q \Vert}_p^q \leqslant C \sum_{{{\Vert k \Vert}_\infty}\leqslant N} \frac{{\varepsilon}^q}{(2N+1)^d} &=& C \varepsilon^q
\end{eqnarray}
Now form(\ref{eq6}), (\ref{eq8}) in (\ref{eq5}) we get \
\begin{eqnarray*}
{\Vert (G_a -1)f \Vert}_{\LP}^q & \leqslant & M_0^q {\varepsilon}^q +C {\varepsilon}^q 
\end{eqnarray*}
Hence \[ \lim_{a\rightarrow 0}{\Vert (G_a -1)f \Vert}_{\LP} =0 \]
Now from Lemma \ref{lem2} that for any $ f \in \LP $ , $ 1 \leqslant p,q < \infty $ ,\
\[\lim_{(a,b)\rightarrow (0,0)}  \Vert S_{a,b;g,\gamma}f-f \Vert_\LP = 0 \]
This proves (\ref{eq1}) for $ 1 \leqslant p,q < \infty $.\\
For $ q = \infty $ with the obvious modification similar result can be obtained. Now we show that (\ref{eq1}) is not true for $p= \infty $ by producing a counter example.
\begin{example} 
For simplicity, we consider only the case of $d=1$. Take some $E \subset [0,1] $ such that $E$ is nowhere dense and is of positive measure. Let $g=\gamma=\chi_E$. For any $a>0$, we have 
\[ \{ x \in [0,1] : G_a(x)>0 \} = \bigcup\limits_{ n \in \BZ} (na+E)\cap [0,1]= \bigcup\limits_{ \Vert n \Vert_{\infty} \leq \frac{1}{a} } (na+E)\cap[0,1]. \]
Since each of $na+E$ is nowhere dense, so is $\bigcup\limits_{ \Vert n \Vert_{\infty} \leq \frac{1}{a} } (na+E)$. Therefore, $ \{ x \in [0,1] : G_a(x)>0 \} $ is nowhere dense. Hence \[ | \{ x \in [0,1] : |G_a(x) -1|=1 \}| \geq |\{ x \in [0,1] : G_a(x)=0 \}| >0. \]
Let $ f_0 = \chi_{[0,1]}. $ Then we have $\Vert(G_a -1)f_0 \Vert_\infty \geq 1,\;\; \forall a>0.$\\
Now \[ \Vert(G_a -1)f_0 \Vert_{W(L^\infty, \ell^q)}^q = \sum_{k \in \BZ} \Vert (G_a-1)\chi_{[0,1]} \cdot \chi_{[k,k+1]} \Vert_{\infty}^q \geq \Vert (G_a-1)\chi_{[0,1]} \Vert_{\infty}^q \] That is, \[ \Vert(G_a -1)f_0 \Vert_{W(L^\infty, \ell^q)} \geq 1, \;\; \forall a>0. \]
By Lemma \ref{lem2},$ \lim\limits_{(a,b)\rightarrow (0,0)} \Vert S_{a,b;g,\gamma}f_0-f_0 \Vert_{W(L^\infty, \ell^q)} \geq 1. $ That is,(\ref{eq1}) fails for $p=\infty.$

Since $ \overline{g} \cdot \gamma $ is locally Riemann integrable, we see from Lemma \ref{lem3} that \[ \lim_{(a,b) \rightarrow (0,0)} \Vert G_a -1\Vert_{\infty} =0.\]
Using Lemma \ref{lem2} again, we get
\[\lim_{(a,b)\rightarrow (0,0)} \Vert S_{a,b;g,\gamma}-I \Vert_{\LP \rightarrow \LP } = 0. \;\;\;\;\;\;\;\;\; \square \]    
\end{example}
\begin{rmk}
Before this result there were few results on weak convergence of Gabor expansions (\cite{fei, str, zim, hei, oko, sun}) but no results on the convergence of Gabor expansions in the operator norm on Wiener amalgam spaces. So this is a new interesting result in this topic. Also as $p=q$ we have $ W(L^p, \ell^p) = L^p(\BR^d)$, so this result extends Sun's (\cite{sun}) result.
\end{rmk}

\section{Important results on Gabor frame operator}
In this section we present Janssen's representation of Gabor frame operators $S_{a,b;g,\gamma}$ and the biorthogonality condition of Wexler-Raz on $\LP$. In \cite{wex} J. Wexler and S. Raz obtained  biorthogonality condition and in \cite{jen} Janssen presented his representation for Gabor frame operator. The relevance of this identity for the study of Gabor frame have pointed out in \cite{fei0, luf, zim, hei}. Now we will establish those result in new setting. To achieve this goal we discuss the necessary theory required, as in (\cite{gro}, page 130).
 
First we expand $G_{a,b;n}$ into its Fourier series. The $l$-th Fourier coefficient of $G_{a,b;n}$ is
\begin{eqnarray*}
\hat{G}_{a,b;n}(l) &=& a^{-d} \int_{Q_a} G_{a,b;n}(x)e^{- 2 \pi i \langle l,x/a \rangle} dx\\
&=& a^{-d} \int_{Q_a} \sum_{k \in \BZ^d}(T_{\frac{n}{b}}\overline{g}\cdot\gamma)(x-ak)e^{- 2 \pi i \langle l,x/a \rangle} dx\\
&=& a^{-d} \int_{\BR^d}(T_{\frac{n}{b}}\overline{g}\cdot\gamma)(x)e^{- 2 \pi i \langle l,x/a \rangle} dx\\
&=& a^{-d} \langle \gamma, M_{\frac{l}{a}}T_{\frac{n}{b}}g \rangle.
\end{eqnarray*}
Since $G_{a,b;n} \in L^{\infty}(Q_a)\subseteq L^{2}(Q_a)$ by Lemma \ref{lem1}, $G_{a,b;n}$ has the Fourier series 
\begin{equation}\label{eq9}
G_{a,b;n}(x)=a^{-d}\sum_{l \in \BZ^d} \langle \gamma, M_{\frac{l}{a}}T_{\frac{n}{b}}g \rangle e^{2 \pi i \langle l,x/a \rangle}
\end{equation}
with convergence in $L^2(Q_a)$. Now substitute this into Walnut's representation (\ref{eq10}), we obtain
\begin{eqnarray*}
S_{a,b;g,\gamma}f &=& \frac{a^d}{\langle \gamma, g \rangle} \sum_{n \in \mathbb{Z}^d} G_{a,b;n} \cdot T_{\frac{n}{b}}f \\
&=& \frac{1}{\langle \gamma, g \rangle} \mathop{\sum\sum}_{l, n \in \BZ^d}\langle \gamma, M_{\frac{l}{a}}T_{\frac{n}{b}}g \rangle M_{\frac{l}{a}}T_{\frac{n}{b}}f, 
\end{eqnarray*} 
in operator notation,
\begin{equation}\label{eq11}
S_{a,b;g,\gamma}=\frac{1}{\langle \gamma, g \rangle} \mathop{\sum\sum}_{l,n \in \BZ^d}\langle \gamma, M_{\frac{l}{a}}T_{\frac{n}{b}}g \rangle M_{\frac{l}{a}}T_{\frac{n}{b}}.
\end{equation}

Now the question is, when this series will converge. If $g,\gamma \in L^2$ then it is not clear how the Fourier series in (\ref{eq9}) represents $G_{a,b;n}$. To avoid this we will take $g,\gamma$ in such a way that the series (\ref{eq9}) and (\ref{eq11}) will converge definitely. In the investigations of Gabor families by Tolimieri-Orr \cite{tol}, and Janssen \cite{jen}, the following technical condition on windows $g, \gamma$ were used. Here we write that condition as definition.  
\begin{defn}
A pair of window functions $(g,\gamma)$  in $L^2(\BR^d)$ satisfies condition $($A' $)$ for the parameters $a,b >0$ if
\begin{equation}\label{eq12}
\sum_{l,n \in \BZ^d} |\langle \gamma, M_{\frac{l}{a}}T_{\frac{n}{b}}g \rangle |< \infty
\end{equation}
If $g=\gamma$, then $g$ is said to satisfy condition $($A$)$ for the parameters $a,b>0$ if
\begin{equation}\label{eq13}
\sum_{l,n \in \BZ^d} |\langle g, M_{\frac{l}{a}}T_{\frac{n}{b}}g \rangle |< \infty
\end{equation}
\end{defn}
Now the condition $($A' $)$ guarantees the absolute convergence of the series expansions (\ref{eq9}) and (\ref{eq11}). However condition $($A' $)$ is not always satisfied even for $g,\gamma \in W(\BR^d)$ (an example given in \cite{gro}, p.132). So we have to put more condition on window function. If we consider $g,\gamma $ as in Feichtinger's algebra $S_{0}(\BR^d)$, then the condition ($A$') is satisfied together for all $a,b>0.$ Now with this hypothesis we derive representation of the Gabor frame operator $S_{a,b;g,\gamma}$. A version of Janssen's representation can be found in \cite{gry}. However a similar result for the Janssen's representation for frame operator $S_{a,b;g,\gamma}$ is proved by taking the window functions $g,\gamma \in S_0(\BR^d)$ in the following theorem.
\begin{thm}
(Janssen's representation) Suppose that $g,\gamma \in S_0(\BR^d)$. Then for all $a,b >0$ frame operator $S_{a,b;g,\gamma}$ can be expressed as follows:
\begin{eqnarray*}
S_{a,b;g,\gamma} &=& \frac{1}{\langle \gamma, g \rangle} \mathop{\sum\sum}_{l, n \in \BZ^d} \langle \gamma, M_{\frac{l}{a}}T_{\frac{n}{b}}g \rangle M_{\frac{l}{a}}T_{\frac{n}{b}}\\
&=& \frac{1}{\langle \gamma, g \rangle} \mathop{\sum\sum}_{k,n \in \BZ^d}\langle \gamma, T_{\frac{k}{b}}M_{\frac{n}{a}}g \rangle T_{\frac{k}{b}}M_{\frac{n}{a}}
\end{eqnarray*}
and converges absolutely in the operator norm.
\end{thm}
\begin{proof}
Let $\tilde{S}_{a,b;g,\gamma} := \frac{1}{\langle \gamma, g \rangle} \mathop{\sum\sum}\limits_{l, n \in \BZ^d} \langle \gamma, M_{\frac{l}{a}}T_{\frac{n}{b}}g \rangle M_{\frac{l}{a}}T_{\frac{n}{b}}$ and we want show $S_{a,b;g,\gamma}=\tilde{S}_{a,b;g,\gamma}$. As $g,\gamma \in S_0(\BR^d)$, so by condition (A') the series for $\tilde{S}_{a,b;g,\gamma}$ converges absolutely in operator norm and hence its expression is independent of the order of summations. Therefore
\begin{eqnarray*}
\tilde{S}_{a,b;g,\gamma} &=& \frac{a^d}{\langle \gamma, g \rangle} \sum_{n \in \BZ^d} \left(a^{-d} \sum_{l \in \BZ^d} \langle \gamma, M_{\frac{l}{a}}T_{\frac{n}{b}}g \rangle e^{2 \pi i \langle l,x/a \rangle} \right) T_{\frac{n}{b}}\\
&=& \frac{a^d}{\langle \gamma, g \rangle} \sum_{n \in \BZ^d} G_{a,b;n} \cdot T_{\frac{n}{b}}=S_{a,b;g,\gamma}
\end{eqnarray*}
by using (\ref{eq9}) and Walnut's representation.
\end{proof}
Next we state and prove the Wexler-Raz biorthogonality condition for frame operator $S_{a,b;g,\gamma}$ on $\LP$. In \cite{wex} the authors found a exceptional relation between window $g$ and dual window $\gamma$. Their conditions characterize all dual windows. Here we make use of Theorem \ref{thm1} and present a version of that important result.
\begin{thm}\label{th5}
(Wexler-Raz biorthogonality)Assume $g,\gamma \in S_0(\BR^d)$. Then for any $ 1 \leq p < \infty$ and $ 1 \leq q \leq \infty$ the following conditions are equivalent:

(i) $ \lim\limits_{(a,b)\rightarrow (0,0)}S_{a,b;g,\gamma}f=\lim\limits_{(a,b)\rightarrow (0,0)}S_{a,b;\gamma,g}f=f$ on $\LP$.

(ii) $\lim\limits_{(a,b)\rightarrow (0,0)} \frac{1}{\langle \gamma, g \rangle} \langle \gamma, M_{\frac{l}{a}}T_{\frac{n}{b}}g \rangle=\delta_{l0} \delta_{n0}$ for $l,n \in \BZ^d.$
\end{thm}
\begin{proof}
$(i)\Rightarrow (ii)$ Let $f \in \LP$ and $h \in W(L^{p'},\ell^{q'})$ (K\"othe dual of $\LP$) and assume that $\lim\limits_{(a,b)\rightarrow (0,0)}S_{a,b;g,\gamma}f=f$. Let $l,m \in \BZ^d$ be  arbitrary. Then  
\begin{eqnarray*}
\delta_{lm}\langle f,h \rangle &=& \langle \lim\limits_{(a,b)\rightarrow (0,0)} S_{a,b;g,\gamma} T_{\frac{l}{b}}f, T_{\frac{m}{b}}h  \rangle\\
&=& \lim\limits_{(a,b)\rightarrow (0,0)} \frac{a^d}{\langle \gamma,g \rangle} \langle \sum_{n \in \mathbb{Z}^d} G_{a,b;n} \cdot T_{\frac{n+l}{b}}f, T_{\frac{m}{b}}h  \rangle\\
&=& \lim\limits_{(a,b)\rightarrow (0,0)} \frac{a^d}{\langle \gamma,g \rangle} \langle G_{a,b; m-l} \cdot T_{\frac{m}{b}}f, T_{\frac{m}{b}}h \rangle\\
&=& \lim\limits_{(a,b)\rightarrow (0,0)} \frac{a^d}{\langle \gamma,g \rangle} \langle (T_{-\frac{m}{b}}G_{a,b;m-l})f,h \rangle
\end{eqnarray*}
So we conclude that $\lim\limits_{(a,b)\rightarrow (0,0)} \frac{a^d}{\langle \gamma,g \rangle}G_{a,b;m-l}(\cdot + \frac{m}{b})f=\delta_{lm}f$. Now varying $l,m \in \BZ^d$, we get $\lim\limits_{(a,b)\rightarrow (0,0)} \frac{a^d}{\langle \gamma,g \rangle}G_{a,b;0}(\cdot)f=f$ and $\lim\limits_{(a,b)\rightarrow (0,0)}\frac{a^d}{\langle \gamma,g \rangle} G_{a,b;n}(\cdot)f=0$ when $n\neq 0$. We already know from (\ref{eq9}) that $G_{a,b;n}$ has the Fourier series
\begin{eqnarray*}
G_{a,b;n}(x) &=& a^{-d}\sum_{l \in \BZ^d} \langle \gamma, M_{\frac{l}{a}}T_{\frac{n}{b}}g \rangle e^{2 \pi i \langle l,x/a \rangle} \\ 
 \Rightarrow \lim\limits_{(a,b)\to (0,0)}\frac{a^d}{\langle \gamma,g \rangle} G_{a,b;n}(x)f(x) &=& \lim\limits_{(a,b)\to (0,0)}\frac{1}{\langle \gamma,g \rangle}\sum_{l \in \BZ^d} \langle \gamma, M_{\frac{l}{a}}T_{\frac{n}{b}}g \rangle e^{2 \pi i \langle l,x/a \rangle}f(x)
\end{eqnarray*}
We conclude by using uniqueness of Fourier coefficients that
\[ \lim\limits_{(a,b)\rightarrow (0,0)} \frac{1}{\langle \gamma, g \rangle} \langle \gamma, M_{\frac{l}{a}}T_{\frac{n}{b}}g \rangle=\delta_{l0} \delta_{n0} .\]
The implication $(ii)\Rightarrow (i)$ follows form Janssen's representation: If the biorthogonality condition (ii) is satisfied
then 
\[ \lim\limits_{(a,b)\rightarrow (0,0)} \sum_{l,n \in \BZ^d}|\langle \gamma, M_{\frac{l}{a}}T_{\frac{n}{b}}g \rangle|=\sum_{l,n \in \BZ^d} | \langle \gamma, g \rangle \delta_{l0} \delta_{n0}|= |\langle \gamma, g \rangle | < \infty, \]
i.e., for arbitrary small $a,b>0$, $\sum\limits_{l,n \in \BZ^d}|\langle \gamma, M_{\frac{l}{a}}T_{\frac{n}{b}}g \rangle| < \infty $, i.e., the pair $(g,\gamma)$ satisfies condition ($A$') for arbitrary small $a,b>0$ and hence the representation (\ref{eq11}) converges in the operator norm. Hence by Theorem \ref{thm1}, for any $ 1 \leq p < \infty$ and $ 1 \leq q \leq \infty$, $\lim\limits_{(a,b)\rightarrow (0,0)}S_{a,b;g,\gamma}f=f$ on $\LP$.
\end{proof}
The next corollary follows immediately from Theorem \ref{th5}.
\begin{cor}
In the assumption of Theorem \ref{th5}, if  $\overline{g} \cdot \gamma$ is locally Riemann integrable then for any  $ 1 \leqslant p,q \leqslant \infty,$ the following conditions are equivalent:

(i) $ \lim\limits_{(a,b)\rightarrow (0,0)}S_{a,b;g,\gamma}=\lim\limits_{(a,b)\rightarrow (0,0)}S_{a,b;\gamma,g}=I$ on $B(\LP)$.

(ii) $\lim\limits_{(a,b)\rightarrow (0,0)} \frac{1}{\langle \gamma, g \rangle} \langle \gamma, M_{\frac{l}{a}}T_{\frac{n}{b}}g \rangle=\delta_{l0} \delta_{n0}$ for $l,n \in \BZ^d.$ 
\end{cor}
\begin{proof}
$(i) \Rightarrow (ii)$ follows from the fact that $\lim\limits_{(a,b)\rightarrow (0,0)}S_{a,b;g,\gamma}=I$ on $B(\LP)$ $\Rightarrow \lim\limits_{(a,b)\rightarrow (0,0)}S_{a,b;g,\gamma}f=f$ on $\LP$.\\
The implication $(ii)\Rightarrow (i)$ follows form Janssen's representation and Theorem \ref{thm1}: If the biorthogonality condition (ii) is satisfied then the pair $(g,\gamma)$ satisfies condition ($A$') for arbitrary small $a,b>0$ and the frame operator $S_{a,b;g,\gamma}$ converges in the operator norm. Since $\overline{g} \cdot \gamma$ is locally Riemann integrable then by using Theorem \ref{thm1} we conclude that for any  $ 1 \leqslant p,q \leqslant \infty,$ 
$ \lim\limits_{(a,b)\rightarrow (0,0)}S_{a,b;g,\gamma}=\lim\limits_{(a,b)\rightarrow (0,0)}S_{a,b;\gamma,g}=I$ on $B(\LP)$.
\end{proof}


\section*{Acknowledgments}
The author is deeply indebted to Prof. Hans G. Feichtinger for several valuable suggestions concerning Theorem \ref{th5} and for important references in this context. The author wish to thank Dr. J. Swain for several fruitful discussions. The author also wishes to thank the Ministry of Human Resource Development, India for the  research fellowship and Indian Institute of Technology Guwahati, India for the support provided during the period of this work. Further, the author thanks the anonymous referee for very valuable suggestions which helped to improve the paper.

\end{document}